\documentclass[a4paper]{article}
\usepackage[font=sf, labelfont={sf,bf}, margin=1cm]{caption}
\usepackage[utf8]{inputenc}
\usepackage{lmodern}
\usepackage{amssymb,amsfonts,amsthm,amsmath,bbm,mathrsfs,aliascnt,mathtools}
\usepackage[english]{babel}
\usepackage[toc,page]{appendix} 
\usepackage{stmaryrd}
\usepackage{verbatim}
\usepackage{hyperref}
\usepackage{color}
\usepackage{bbm}
\usepackage[T1]{fontenc}
\usepackage[cyr]{aeguill}
\date{}

\title{\bf{\textsc{Ergodic aspects of some \\Ornstein-Uhlenbeck type processes\\ related to L\'evy processes}}}

\author{Jean Bertoin\thanks{ Institut f\"ur Mathematik, 
Universit\"at Z\"urich, 
Winterthurerstrasse 190, 
CH-8057 Z\"urich, Switzerland. \hfill \eject
Email: jean.bertoin@math.uzh.ch}  }

\begin{document}

\maketitle

\newtheorem{theo}{Theorem}[section]
\newtheorem{lemma}{Lemma}[section]
\newtheorem{prop}{Proposition}[section]
\newtheorem{cor}{Corollary}[section]
\newtheorem{defi}{Definition}[section]
\newtheorem{rem}{Remark}
\numberwithin{equation}{section}
\newcommand{\e}{{\mathrm e}}
\newcommand{\rep}{{\mathrm{rep}}}
\newcommand{\R}{{\mathbb{R}}}
\newcommand{\C}{{\mathbb{C}}}
\newcommand{\T}{\mathbf{T}}
\newcommand{\E}{{\mathbb{E}}}
\renewcommand{\L}{\mathbf{L}}
\newcommand{\TT}{\mathbb{T}}
\newcommand{\dd}{\mathrm{d}}
\newcommand{\p}{\mathcal{P}}
\newcommand{\N}{\mathbb{N}}
\newcommand{\D}{\mathcal{D}}
\newcommand{\W}{\mathcal{W}}
\newcommand{\pr}{\mathbb{P}}
\newcommand{\Q}{\mathbb{Q}}
\newcommand{\z}{\mathbf{\zeta}}
\newcommand{\m}{\mathbf{\nu}}
\newcommand{\M}{\mathcal{M}}
\newcommand{\Z}{\mathbb{Z}}
\newcommand{\dia}{\diamond}
\newcommand{\bul}{\bullet}
\renewcommand{\geq}{\geqslant}
\renewcommand{\leq}{\leqslant}
\newcommand{\veps}{\varepsilon}

\begin{abstract}
This work concerns the Ornstein-Uhlenbeck type process  associated to
 a positive self-similar Markov process $(X(t))_{t\geq 0}$ which drifts to $\infty$, namely $U(t)\coloneqq \e^{-t}X(\e^{t}-1)$.
 We point out that $U$ is always a (topologically) recurrent Markov process and identify its invariant measure in terms of the law of the exponential functional $\hat I\coloneqq \int_0^{\infty}\exp(\hat\xi_s)\dd s$, where $\hat\xi$ is the dual of the real-valued L\'evy process $\xi$ related to $X$ by the Lamperti transformation.  This invariant measure is infinite (i.e. $U$ is null-recurrent) if and only if $\xi_1\not \in L^1(\pr)$. In that case, we determine the family of L\'evy processes $\xi$ for which  $U$ fulfills the conclusions of the Darling-Kac theorem. Our approach relies crucially on a remarkable connection due to Patie \cite{Patie} with
 another generalized Ornstein-Uhlenbeck process that can be associated to the L\'evy process $\xi$, and properties of time-substitutions based on additive functionals.   
\end{abstract}

\medskip


\medskip

\section{Introduction}

 Let $(\xi_t)_{t\geq 0}$ be a real-valued L\'evy process which drifts to $\infty$, that is  $\lim_{t\to \infty} \xi_t=\infty$ a.s. The so-called exponential functional 
 $$I(t)\coloneqq  \int_0^t\exp(\xi_s)\dd s$$
 defines a random bijection $I: \R_+\to \R_+$, and we denote  its inverse by $\tau$. A well-known transformation due to Lamperti \cite{La2},
 $$X(t)\coloneqq \exp(\xi_{\tau(t)}), $$
 yields a  Markov process $(X(t))_{t\geq 0}$ on $(0,\infty)$ that enjoys the scaling property (with index $1$), in the sense that for every $x>0$, $(xX(t/x))_{t\geq 0}$ is a version of $X$ started from $x$. Conversely, any Markov process $X$ on $(0,\infty)$ that fulfills the scaling property (with index $1$) and drifts to $\infty$  can be constructed in this way. We refer to the survey by Pardo and Rivero \cite{ParRiv} and references therein for a detailed presentation of the topic.

 The question of  the existence of  a {\em truly} self-similar  version $(\tilde X_t)_{t\geq 0}$,  that is, 
$\tilde X$ is a Markov process with the same transition probabilities as $X$  and  further there is the identity in distribution 
 $$ (c\tilde X(t/c))_{t\geq 0} \ {\overset{(d)}{=}} \ (\tilde X(t))_{t\geq 0} \qquad \text{ for every } c>0,$$ 
 is  equivalent to the question of whether 
$0+$ is an entrance boundary for the Markov process $X$. This
 was raised by Lamperti, and settled in the present setting\footnote{This question makes also sense when $\xi$ oscillates, that is  $\limsup_{t\to \infty} \xi_t=\infty$ and $\liminf_{t\to \infty} \xi_t=-\infty$ a.s. It was proved in \cite{CabCha} and \cite{CKPR} that the answer is positive if and only if the so-called ascending ladder height of $\xi$ has a finite expectation.} in  \cite{BY1}: the answer is positive if and only if $\xi_1\in L^1(\pr)$ (recall that then $\E(\xi_1)>0$, since the test of Chung and Fuchs ensures that in dimension $1$, centered L\'evy processes are recurrent and therefore oscillate), and further the stationary law can then be expressed in terms of the exponential functional $\hat I$ of the dual L\'evy process.

 On the other hand, there is another well-known transformation {\it \`a la} Ornstein-Uhlenbeck, also due to Lamperti \cite{La1}, that yields a bijection between self-similar processes and stationary processes. In the present setting, assuming again that $\xi_1\in L^1(\pr)$ and writing $\tilde X$ for the self-similar version of $X$, 
 $$\tilde U(t)\coloneqq \e^{-t}\tilde X(\e^t), \qquad t\in \R$$
 is a stationary process on $(0,\infty)$. Furthermore, the scaling property ensures that 
  \begin{equation}\label{e:defU}
  U(t)\coloneqq \e^{-t}X(\e^t-1), \qquad t\geq 0
  \end{equation} 
  is Markovian,  and the Markov processes $U$ and $\tilde U$ have the same semigroup.

 The initial motivation for this work is to analyze the  situation when $\xi_1\not\in L^1(\pr)$.  We shall show that the Ornstein-Uhlenbeck type process $U$ 
  still possesses a stationary version $\tilde U$, but now under an infinite measure
 $\Q$ which is absolutely continuous with respect to $\pr$. 
 More precisely,  $U$ is a (null) recurrent Markov process and its invariant measure $\nu$ can be expressed similarly as in the positive recurrent case in terms of the dual exponential functional $\hat I$. 
 When $\E(\xi_1)=\infty$, the claim that $U$ is recurrent might look surprising at first sight, since the L\'evy process may grow faster than any given polynomial (think for instance of stable subordinators). One could expect that the same might hold for $X(t)=\exp(\xi_{\tau(t)})$, which would then impede  the recurrence of $U$. However the time-substitution by $\tau$ has a slowing down effect when $X$ gets larger, and actually $X$  only grows linearly fast.  
 
 Our main result is related to the celebrated Darling-Kac theorem, which can be thought of as a version of Birkhoff's ergodic theorem in infinite invariant measure; see e.g. Theorem 3.6.4 in \cite{Aaron} and Theorem 8.11.3 in \cite{BGT}. We show that if $a:(0,\infty)\to (0,\infty)$ is regularly varying at $\infty$ with index $\alpha\in(0,1)$,
 then for every nonnegative $f\in L^1(\nu)$,
 $a(t)^{-1} \int_0^t f(U(s)) \dd s$ converges in distribution as $t\to \infty$ towards a Mittag-Leffler distribution with parameter $\alpha$ if and only if $b(t)^{-1}\xi_t$ converges in distribution as $t\to \infty$ to a positive stable random variable with exponent $\alpha$, where $b$ denotes an asymptotic inverse of $a$.

At the heart of our approach lies the fact that one can associate to the L\'evy process $\xi$ another generalized 
 Ornstein-Uhlenbeck process, namely
 $$V(t)\coloneqq \exp(-\xi_t) \left(I(t)+V(0)\right), \qquad t\geq 0.$$ Lindner and Maller \cite{LinMa} have shown that, since $\xi$ drifts to $\infty$,  $V$
 always possesses a stationary version $\tilde V$, no matter whether $\xi_1$ is integrable or not. Patie \cite{Patie} pointed at a remarkable connection between $U$ and $V$  via a simple time
 substitution, and this provides a powerful tool for the analysis of  $U$.
 
 The rest of this paper is organized as follows. We start in Section 2 by providing background on the generalized Ornstein-Uhlenbeck process $V$. Then, in Section 3, we construct a stationary version $\tilde U$ of $U$ under a possibly infinite equivalent measure, and point at the topological recurrence of $U$. Finally, in Section 4, we address the Darling-Kac theorem for the occupation measure of $\tilde U$. 
 
 \section{Background on another  generalized Ornstein-Uhlenbeck process}
 
 We start by recalling the basic time-reversal property  of L\'evy processes, also known as the duality identity, which  
 plays an important role in this subject. If we denote $\hat \xi$ the so-called dual L\'evy process which has the same law as $-\xi$, then 
 for every $t>0$, there is the identity in distribution between c\`adl\`ag processes 
$$(-\xi_t+\xi_{(t-s)-})_{0\leq s \leq t}\ {\overset{(d)}{=}}\ (\hat \xi_s)_{0\leq s \leq t}.$$

Following Carmona {\it et al.} \cite{CPY} and Lindner and Maller \cite{LinMa}, as well as other authors, we  associate to the L\'evy process $\xi$   another generalized Ornstein-Uhlenbeck process $(V(t))_{t\geq 0}$, 
 $$V(t)\coloneqq  \exp(-\xi_t)\left( I(t)+V(0)\right)=  \int_0^t \exp(\xi_s-\xi_t)\dd s +V(0)\exp(-\xi_t),$$
 where the initial value $V(0)$ is arbitrary and may be random.
It was   observed in \cite{CPY} and  \cite{LinMa} that the time-reversal property and the a.s. finiteness of the dual exponential functional
$$\hat I\coloneqq \hat I(\infty)= \int_0^{\infty}\exp(\hat \xi_s)\dd s$$
(which is known to follow from our assumption that $\xi$ drifts to $\infty$, see Theorem 1 in \cite{BY2}, or Theorem 2 in \cite{ErMa}),
 immediately implies that
 \begin{equation}\label{e:vconv}
 \lim_{t\to \infty}V(t) = \hat I\qquad \text{in distribution,}
 \end{equation}
  independently of the initial value $V(0)$.  The distribution of $\hat I$,
  $$\mu(\dd x)\coloneqq \pr(\hat I \in \dd x), \qquad x\in (0,\infty),$$
   thus plays a fundamental role in this setting; it has been studied in depth in the literature, see in particular \cite{BehLin,KPS,MauZw,Riv} and references therein.
  
   Lindner and Maller (Theorem 2.1 in \cite{LinMa}) pointed at the fact that if $V(0)$ has the same law as $\hat I$ and is independent of $\xi$, then the process $(V_t)_{t\geq 0}$ is stationary.
It will be convenient for us to rather work with a two-sided version  $(\tilde V_t)_{t\in \R}$ which can easily be constructed as follows.

Assume henceforth that $(\hat \xi_t)_{t\geq 0}$ is an independent copy of $(-\xi_t)_{t\geq 0}$, and write $(\tilde \xi_t)_{t\in \R}$ for the two-sided L\'evy process given by
$$\tilde \xi_t= \left\{ 
\begin{matrix} 
\xi_t & \text{ if }t\geq 0,\\
\hat\xi_{|t|-} & \text{ if }t< 0.\\
 \end{matrix} \right.$$
 We then set for every $t\in \R$
 $$\tilde I(t)\coloneqq \int_{-\infty}^t \exp(\tilde \xi_s)\dd s\quad \text{ and } \quad \tilde V(t)= \exp(-\tilde \xi_t) \tilde I(t).$$
 Note that $\tilde V(0)=\tilde I(0)=\hat I$,  so the process $(\tilde V_t)_{t\geq 0}$ is a version of $V$ started from its stationary distribution. The next statement records some important properties of $\tilde V$ that will be useful for this study.

 \begin{theo} \label{T1} \begin{enumerate}
\item[(i)] The process $(\tilde V_t)_{t\in \R}$ is a stationary and strongly mixing Feller process, with stationary one-dimensional distribution $\mu$.
\item[(ii)] For every $f\in L^1( \mu)$, we have 
 $$\lim_{t\to \infty} \frac{1}{t}\int_0^t f(\tilde V(s)) \dd s = \langle \mu,f\rangle \quad \text{a.s.}$$
 \end{enumerate}
 \end{theo}
\begin{proof} (i)    By the time-reversal property, the two-sided process $\tilde \xi$ has stationary increments, in the sense that for every $t\in \R$, $(\tilde \xi_{t+s}-\tilde \xi_t)_{s\in \R}$ has the same law as $(\tilde \xi_s)_{s\in \R}$.  This readily entails the stationarity of $\tilde V$.
The Feller property has already been pointed at in Theorem 3.1 in \cite{BehLin}, so it only remains to justify the strong mixing assertion. Unsurprisingly\footnote{If  the Markov process $V$ is $\mu$-irreducible, then one can directly apply well-known facts about stochastic stability; see Part III in Meyn and Tweedie \cite{MeynT}. However, establishing irreducibility  for arbitrary generalized Ornstein-Uhlenbeck processes seems to be a challenging task; see Section 2.3 in Lee \cite{Lee} for a partial result.}, this follows from \eqref{e:vconv} by a monotone class argument that we recall for completeness.

Let  ${\mathcal L}^{\infty}$ denote the space of bounded  measurable functions $g: (0,\infty)\to \R$ and 
${\mathcal C_b}$ the subspace of continuous bounded functions. Introduce the vector space
$${\mathcal H}\coloneqq \{g\in {\mathcal L}^{\infty}: \lim_{t\to \infty} \E(f(\tilde V(0))g(\tilde V(t)))=\langle \mu,f\rangle \langle \mu,g\rangle\text{ for every }f\in {\mathcal L}^{\infty}\}.$$
We easily deduce from \eqref{e:vconv} that ${\mathcal C_b}\subseteq {\mathcal H}$. Then consider a non-decreasing sequence $(g_n)_{n\in \N}$ in ${\mathcal H}$ with $\sup_{n\in\N} \| g_n\|_{\infty}<\infty$ and let $g= \lim_{n\to \infty} g_n$. For every $f\in {\mathcal L}^{\infty}$, we have by stationarity 
$$\E(f(\tilde V(0))g(\tilde V(t)))= \E(f(\tilde V(-t))g(\tilde V(0))).$$
So assuming for simplicity that $\|f\|_{\infty}\leq 1$, 
the absolute difference
$$|\E(f(\tilde V(0))g(\tilde V(t)))-\langle \mu,f\rangle \langle \mu,g\rangle|$$
  can be bounded from above by 
$$
 \E( g(\tilde V(0)))-g_n(\tilde V(0))) +  \langle \mu,g-g_n\rangle 
+|\E(f(\tilde V(0))g_n(\tilde V(t)))-\langle \mu,f\rangle \langle \mu,g_n\rangle| .
$$
The first two terms in the sum above coincide and can be made as small as we wish by choosing $n$ large enough. Since $g_n\in{\mathcal H}$, this entails
$$\lim_{t\to \infty} |\E(f(\tilde V(0))g(\tilde V(t)))-\langle \mu,f\rangle \langle \mu,g\rangle| \leq \varepsilon$$
for every $\varepsilon>0$. Hence $g\in{\mathcal H}$, and since ${\mathcal C_b}$ is an algebra that contains the constant functions, we conclude by a functional version of the monotone class theorem that ${\mathcal H}={\mathcal L}^{\infty}$.

(ii) Since strong mixing implies ergodicity, this follows from Birkhoff's ergodic theorem. 
\end{proof}

 We mention that the argument for Theorem \ref{T1} applies more generally to the larger class of generalized Ornstein-Uhlenbeck processes considered by Lindner and Maller \cite{LinMa}. Further,  sufficient  conditions ensuring exponential ergodicity can be found in Theorem 4.3 of Lindner and Maller \cite{LinMa}, Lee \cite{Lee}, Wang \cite{Wang},  and Kevei \cite{Kevei}.
  
  \section{A time substitution and its consequences}
  Patie \cite{Patie} pointed out that the Ornstein-Uhlenbeck type processes $U$ and $V$ are related by a simple time-substitution. We shall see here that the same transformation, now applied to the stationary  process $\tilde V$, yields a stationary version $\tilde U$ of $U$, and then draw some consequences of this construction.

  Introduce 
  the additive functional
  $$A(t)\coloneqq \int_0^t \frac{\dd s}{\tilde V(s)}=\ln \tilde I(t) - \ln \tilde I(0) \,, \qquad t\in \R;$$
  clearly $A: \R\to \R$ is bijective and we denote the inverse bijection by $T$. Observe that $A(T(t))=t$ yields
   the useful identity
  \begin{equation}\label{e:T(t)}
  \int_{-\infty}^{T(t)} \exp(\tilde \xi_s)\dd s= \tilde I(0) \e^t \qquad \text{for all }t\in\R.
  \end{equation}
  
  We also define a measure $\nu$ on $(0,\infty)$ by
  $$\langle \nu,f\rangle=  \int_{(0,\infty)} \frac{1}{x}f(1/x)  \mu(\dd x),$$ 
  and further  introduce an equivalent sigma-finite measure on the underlying probability space $(\Omega, {\mathcal A}, \pr)$ 
  by
  $$\Q(\Lambda) = \E\left(\frac{1}{\tilde V(0)} {\mathbf 1}_{ \Lambda}\right), \qquad \Lambda\in  {\mathcal A}.$$
  Note that 
  $$\Q(f(1/\tilde V(0)))= \langle \nu,f\rangle$$
  for every measurable $f: (0,\infty)\to \R_+$.

  \begin{theo} \label{T2} 
   \begin{itemize}
  \item[(i)] The measure $\Q$ (respectively, $\nu$) is finite if and only if $\xi_1\in L^1(\pr)$, and in that case, $\Q(\Omega)=\nu((0,\infty))=\E(\xi_1)$.
  
  \item[(ii)] Under $\Q$, 
  $$\tilde U(t)\coloneqq 1/\tilde V(T(t)), \qquad t\in \R$$
  is a stationary and ergodic Markov process, with one-dimensional marginal $\nu$ and the same semigroup as the Ornstein-Uhlenbeck type process $U$ defined in \eqref{e:defU}.
  
  \item[(iii)] For all functions $f,g\in L^1(\nu)$ with $\langle \nu,g\rangle\neq 0$, we have
  $$\lim_{t\to \infty} \frac{\int_0^t f(\tilde U(s))\dd s}{\int_0^t g(\tilde U(s))\dd s}= \frac{\langle \nu,f\rangle}{\langle \nu,g\rangle} \quad\Q\text{-a.s. and therefore also }\pr\text{-a.s.}$$
  \end{itemize}
  \end{theo}
  
  \begin{proof} (i) Recall that $\tilde V(0)=\tilde I(0)=\hat I$, so $\Q(\Omega)=\nu((0,\infty))=\E(1/\hat I)$. When $\xi_1\in L^1(\pr)$, Equation (3) in \cite{BY1} gives $\E(1/\hat I)=\E(\xi_1)$.
  
  Next, suppose that $\xi_1^-\in L^1(\pr)$ and $\xi_1^+\notin L^1(\pr)$, that is the mean $\E(\xi_1)$ exists and is infinite. 
  We can construct by truncation of the large jumps of $\xi$, an increasing sequence $(\xi^{(n)})_{n\in\N}$ of L\'evy processes such that $\xi^{(n)}_1\in L^1(\pr)$ with $\E(\xi^{(n)}_1)>0$ and
  $\lim_{n\to \infty} \xi^{(n)}_t=\xi_t$ for all $t\geq 0$ a.s. In the obvious notation, 
  $\hat I^{(n)}$ decreases to $\hat I$  as $n\to \infty$, and $\lim_{n\to \infty}\E(\xi^{(n)}_1)=\infty$. We conclude by monotone convergence that $\E(1/\hat I)=\infty$.
  
  Finally, suppose that both $\xi_1^-\notin L^1(\pr)$ and $\xi_1^+\notin L^1(\pr)$, so the mean of $\xi_1$ is undefined. Equivalently, in terms of the L\'evy measure, say $\Pi$, of $\xi$, we have 
  $$\int_{(-\infty,-1)}|x| \Pi(\dd x)= \int_{(1,\infty)} x \Pi(\dd x)= \infty\,,$$
  see Theorem 25.3 in Sato \cite{Sato}. Using Erickson's test characterizing L\'evy processes which drift to $\infty$ when the mean is undefined (see Theorem 15 in Doney \cite{Doney}), it is easy to decompose $\xi$
 into the sum $\xi=\xi'+\eta$ of two independent L\'evy processes, such that $\xi'$ is a L\'evy process with infinite mean and  $\eta$ is a compound Poisson process with undefined mean that drifts to $\infty$. The event $\Lambda\coloneqq \{\eta_t\geq 0 \text{ for all }t\geq 0\}$ has a positive
 probability (because $\eta$ is compound Poisson and drifts to $\infty$). On that event, we have  $\xi\geq \xi'$ and thus also, in the obvious notation, $\hat I \leq \hat I'$. 
 This yields  $$\E(1/\hat I, \Lambda) \geq \E(1/\hat I' )\pr(\Lambda),$$
  and we have see above that the first term in the product is infinite. We conclude that $\E(1/\hat I)=\infty$.

  (ii) It is convenient to view now $\Omega$ as the space of c\`adl\`ag paths $\omega: \R\to (0,\infty)$ endowed with the usual shift automorphisms $(\theta_t)_{ t\in\R}$, i.e. 
  $\theta_t(\omega)=\omega(t+\cdot)$, and $\pr$ as the law of $\tilde V$. 
  We have seen in Theorem \ref{T1}(i) that $\pr$ is $(\theta_t)$-ergodic. 
  
  General results  due to Maruyama and Totoki on time changes of flows based on additive functionals show that the measure $\Q$ is invariant for the time-changed flow
   of automorphisms $(\theta'_t)_{ t\in\R}$, where $\theta'_t(\omega)\coloneqq \omega(T(t)+\cdot))$. See Theorems 4.1(iii) and 4.2 in \cite{Toto}. Further, ergodicity is always preserved by such time substitutions, see Theorem 5.1 in \cite{Toto}. This shows that $(\tilde V(T(t)))_{t\geq 0}$ is a stationary ergodic process under $\Q$.
   
   On the other hand,  time substitution based on an additive functional also preserves the strong Markov property, so $(\tilde V(T(t)))_{t\geq 0}$ is a Markov process under $\Q$.
   By stationarity, $(\tilde V(T(t)))_{t\in\R}$ is Markov too. Composing with the inversion $x\mapsto 1/x$, we conclude that $\tilde U$ is a stationary and ergodic Markov process under $\Q$. 
   
   It remains to determine the semigroup of $\tilde U$, and for this, we simply recall from Theorem 1.4 of Patie \cite{Patie}  that the processes $U$ and $V$ can be related by the same time-substitution as that relating $\tilde U$ and $\tilde V$. As a consequence, $\tilde U$ and $U$ have the same semigroup.
    
  (iii) Under $\Q$, this is a consequence of (ii) and Hopf's ratio ergodic theorem. See also Lemma 5.1 in \cite{Toto}. The measures $\pr$ and $\Q$ being equivalent, the statement of convergence also holds $\pr$-a.s. We mention that, alternatively, this can also be deduced from Birkhoff's ergodic theorem for $\tilde V$ (Theorem \ref{T1}(ii)) by change of variables.
  \end{proof} 
  
\begin{rem}
\begin{enumerate}
\item[(i)] In the case $\xi_1\in L^1(\pr)$, Theorem \ref{T2}(i-ii) agrees with the results in \cite{BY1}; the arguments in the present work are however much simpler. We stress that one should not conclude from Theorem \ref{T2}(i-ii) that $U(t)$ then converges in distribution to the normalized version of $\nu$. Actually this fails when the L\'evy process is lattice-valued (i.e. $\xi_t\in r \Z$ a.s. for some $r>0$, think for instance of the case when $\xi$ is a Poisson process), because then the Ornstein-Uhlenbeck type  process $U$ is periodic. 

\item[(ii)]  Inverting the transformation {\it \`a la} Ornstein-Uhlenbeck incites us to set
$$\tilde X (t)\coloneqq t\tilde U(\ln t)= t/\tilde V(T(\ln t)), \qquad t>0,$$
and the calculation in the proof of  Theorem \ref{T2}(ii)  yields the expression {\it \`a la} Lamperti
$$\tilde X(t)\coloneqq \exp(\tilde \xi_{\tilde \tau(t)}), $$
with $\tilde \tau: (0,\infty)\to \R$ the inverse of the exponential functional $\tilde I$.
Theorem \ref{T2}(ii) entails that under $\Q$, $\tilde X$  is a self-similar version $X$. We refer to \cite{BeSa} for an alternative similar construction which does not require working under an equivalent measure.

\item[(iii)] If we write ${\mathcal G}$ for the infinitesimal generator of the Feller process $V$, then the stationary of the law $\mu$ is is characterized by the identity ${\langle \mu,{\mathcal G}f\rangle}=0$ for every $f$ in the domain of ${\mathcal G}$. Informally\footnote{The application of Volkonskii's formula is not legitimate, since the function $x\mapsto 1/x$ is not bounded away from $0$.}, according to a formula of Volkonskii (see (III.21.6) in \cite{RW}), the infinitesimal generator ${\mathcal G}'$ of the time-changed process $V\circ T$ is given by ${\mathcal G}' f(x)=x {\mathcal G}f(x)$,
so the measure $\mu'(\dd x) \coloneqq x^{-1} \mu(\dd x)$ fulfills  ${\langle \mu',{\mathcal G}'f\rangle}=0$ for every $f$ in the domain of ${\mathcal G}$, and thus should be invariant for the time-changed process $ V\circ T$. We then recover the assertion that $\nu$ is invariant for $\tilde U = 1/(\tilde V\circ T)$.

\end{enumerate}
\end{rem}

We conclude this section by discussing recurrence. 
Recall first  that the support of the stationary law $\mu$ of the generalized Ornstein-Uhlenbeck process $V$ is always an interval, say ${\mathcal I}$; see  Haas and Rivero \cite{HaRi} or  Lemma 2.1 in \cite{BLM}.
More precisely, excluding implicitly the degenerate case when $\xi$ is a pure drift, 
${\mathcal I}=[0,1/b]$ if $\xi$ is a non-deterministic subordinator with drift $b>0$, 
${\mathcal I}=[1/b, \infty)$ if $\xi$ is non-deterministic and of finite variation L\'evy process with no positive jumps and 
 drift $b > 0$, and  ${\mathcal I}=[0, \infty)$ in the remaining cases. Writing ${\mathcal I}^o$ for the interior of ${\mathcal I}$, it is further readily checked that $V(t) \in {\mathcal I}^o$ for all $t\geq 0$ a.s. whenever $V(0) \in {\mathcal I}^o$.
 
 \begin{cor} The Ornstein-Uhlenbeck type process $U$ is topologically recurrent, in the sense that for every $x>0$ with $1/x\in {\mathcal I}^o$, $U$ visits every neighborhood of $x$ a.s., no matter its initial value $U(0)$.
 \end{cor}
 \begin{proof}
 It follows from \eqref{e:vconv} and the Portmanteau theorem that every point $x\in {\mathcal I}^o$ is topologically recurrent for the generalized Ornstein-Uhlenbeck process $V$. Plainly, this property is preserved by time-substitution. 
 \end{proof}
 
\section{On the Darling-Kac theorem}
We assume throughout this section that $\xi_1\not \in L^1(\pr)$, so $\nu$ (and also $\Q$) is an infinite measure. Aaronson's ergodic theorem (see, e.g. Theorem 2.4.2 in \cite{Aaron}) states that for every $f\in L^1(\nu)$, $f\geq 0$, and every potential normalizing function $a: \R_+\to (0,\infty)$, one always have either
$$\limsup_{t\to \infty} \frac{1}{a(t)} \int_0^t f(\tilde U(s))\dd s = \infty \qquad \text{a.s.}$$
or 
$$\liminf_{t\to \infty} \frac{1}{a(t)} \int_0^t f(\tilde U(s))\dd s = 0 \qquad \text{a.s.}$$

Without further mention, we shall henceforth implicitly work under the probability measure $\pr$, and 
say that a family $(Y(t))_{t>0}$ of random variables has a non-degenerate limit in distribution as $t\to \infty$  if $Y(t)$ converges in law towards some not a.s. constant random variable. 

Motivated by the famous Darling-Kac's theorem, the purpose of this section is to provide an explicit necessary and sufficient condition in terms of the L\'evy process $\xi$ for the existence of a normalizing function $a: \R_+\to (0,\infty)$ such that  the normalized occupation measure of $U$ converges in distribution as $t\to \infty$ to
a non-degenerate limit. 

We start with the following simple observation. 

\begin{lemma} \label{L1} The following assertions are equivalent
\begin{itemize}
\item[(i)] For every $f\in L^1(\nu)$ with  $\langle \nu,f\rangle\neq0$,
$$\frac{1}{a(t)}\int_0^t f(\tilde U(s))\dd s, \qquad t>0$$
has a non-degenerate limit in distribution as $t\to \infty$.
\item[(ii)] $\left(\frac{T(t)}{a(t)}\right)_{t>0}$ has a  non-degenerate limit in distribution as $t\to \infty$. 
\end{itemize}
\end{lemma}
 \begin{proof} Note that the identity function 
 $g(x)\equiv 1/x$ always belongs to $L^1(\nu)$, actually with $\langle \nu, g\rangle =1$, and 
 $\int_0^t \dd s/\tilde U(s)=T(t)$. The claim thus follows from Hopf's ratio ergodic theorem (Theorem \ref{T2}(iii)) combined with Slutsky's theorem.
 \end{proof}
 
 For the sake of simplicity, we shall focus on the case when the sought normalizing function
  $a:(0,\infty)\to (0, \infty)$ is regularly varying at $\infty$ with index $\alpha\in(0,1)$. Recall from the Darling-Kac theorem (Theorem 8.11.3 in \cite{BGT}) that this is essentially the only situation in which interesting asymptotic behaviors can occur. Recall also from Theorem 1.5.12 in \cite{BGT} that $a$ then possesses an asymptotic inverse $b: (0,\infty)\to (0,\infty)$,
 in the sense that $a(b(t))\sim b(a(t)) \sim t$ as $t\to \infty$, such that
$b$ is regularly varying at $\infty$ with index $1/\alpha$.

We may now state the main result of this work, which specifies the Darling-Kac theorem for Ornstein-Uhlenbeck type processes. 

\begin{theo} \label{TDK} The following assertions are equivalent:
\begin{enumerate}
\item[(i)] $b(t)^{-1} \xi_t$ has a non-degenerate limit in distribution as $t\to \infty$.
\item[(ii)] Let $f\in L^1(\nu)$ with  $\langle \nu,f\rangle\neq0$.  Then 
$$\frac{1}{a(t)}\int_0^t f(\tilde U(s))\dd s, \qquad t>0$$
has a non-degenerate limit in distribution as $t\to \infty$.
\end{enumerate}
In that case, the limit in (i) is a positive $\alpha$-stable variable, say $\sigma$, with
$$\E(\exp(-\lambda \sigma))=\exp(-c\lambda^{\alpha})$$
for some $c>0$, and the limit in (ii) has the law of 
$\langle \nu,f\rangle \sigma^{-\alpha}$ (and is thus proportional to a Mittag-Leffer variable with parameter $\alpha$). 
\end{theo}

\begin{rem}  In the case when $\xi$ is a subordinator,  Caballero and Rivero proved that the assertion (i) in Theorem \ref{TDK} is equivalent to the assertion (i) of Lemma \ref{L1}
with the weak limit there given by a Mittag-Leffler distribution; see Proposition 2 in \cite{CabRiv}. Thus in that special case, Theorem \ref{TDK} follows directly from  Proposition 2 in \cite{CabRiv} and the present Lemma \ref{L1}.
\end{rem}
\begin{proof} Assume (i); it is well-known that the non-degenerate weak limit  $\sigma$ of $b(t)^{-1}\xi_t$ is an $\alpha$-stable variable, which is necessarily positive a.s. since $\xi$ drifts to $\infty$. Recall that $A(t)=\ln \tilde I(t)-\ln \tilde I(0)$ and write
$$\xi_t = \ln \tilde I(t) -  \ln \tilde V(t).$$ 
We deduce from the stationarity of $\tilde V$ and 
  Slutsky's theorem that there is the weak converge
  $$b(t)^{-1}A(t) \Longrightarrow \sigma \qquad \text{as }t\to \infty.$$
  Using the assumption that $b$ is an asymptotic inverse of $a$ and recalling that $b$ is regularly varying with index $1/\alpha$, this entails by a standard argument that
  $$a(t)^{-1}T(t) \Longrightarrow \sigma^{-\alpha} \qquad \text{as }t\to \infty,$$
  and we conclude from Lemma \ref{L1} (it is well-known that $\sigma^{-\alpha}$ is proportional to a Mittag-Leffler variable with parameter $\alpha$; see for instance Exercise 4.19 in Chaumont and Yor \cite{ChYo}).
  
  Conversely, if (ii) holds for some $f\in L^1(\nu)$ with  $\langle \nu,f\rangle\neq0$, then by Hopf's ergodic theorem and Lemma \ref{L1},   
  $$a(t)^{-1}T(t) \Longrightarrow G\qquad \text{as }t\to \infty,$$
  for some non-degenerate random variable $G$. The same argument as above yields
  $$b(t)^{-1}\xi_t \Longrightarrow G^{-1/\alpha} \qquad \text{as }t\to \infty,$$
  and $G^{-1/\alpha} $ has to be a positive $\alpha$-stable variable.
  \end{proof}
 More precisely,  the argument of the proof shows that  when (i) is satisfied,  the weak convergence in (ii) holds
  independently of the initial value $\tilde U(0)$. That is, equivalently, one may replace $\tilde U$ by $U$, the starting point $U(0)$ being arbitrary. 

\vskip 2mm {\bf Acknowledgment:} I would like to thank V{\'\i}ctor  {Rivero} for pointing at important references which I missed in the first draft of this work.
  
\bibliographystyle{plain}
\bibliography{ergOU}

\begin{thebibliography}{10}

\bibitem{Aaron}
Jon Aaronson.
\newblock {\em An introduction to infinite ergodic theory}, volume~50 of {\em
  Mathematical Surveys and Monographs}.
\newblock American Mathematical Society, Providence, RI, 1997.

\bibitem{BehLin}
Anita Behme and Alexander Lindner.
\newblock On exponential functionals of {L}\'evy processes.
\newblock {\em J. Theoret. Probab.}, 28(2):681--720, 2015.

\bibitem{BLM}
Anita Behme, Alexander Lindner, and Makoto Maejima.
\newblock On the range of exponential functionals of {L}\'evy processes.
\newblock In {\em S\'eminaire de {P}robabilit\'es {XLVIII}}, volume 2168 of
  {\em Lecture Notes in Math.}, pages 267--303. Springer, Cham, 2016.

\bibitem{BeSa}
Jean Bertoin and Mladen Savov.
\newblock Some applications of duality for {L}\'evy processes in a half-line.
\newblock {\em Bull. Lond. Math. Soc.}, 43(1):97--110, 2011.

\bibitem{BY1}
Jean Bertoin and Marc Yor.
\newblock The entrance laws of self-similar {M}arkov processes and exponential
  functionals of {L}\'evy processes.
\newblock {\em Potential Anal.}, 17(4):389--400, 2002.

\bibitem{BY2}
Jean Bertoin and Marc Yor.
\newblock Exponential functionals of {L}\'evy processes.
\newblock {\em Probab. Surv.}, 2:191--212, 2005.

\bibitem{BGT}
N.~H. Bingham, C.~M. Goldie, and J.~L. Teugels.
\newblock {\em Regular variation}, volume~27 of {\em Encyclopedia of
  Mathematics and its Applications}.
\newblock Cambridge University Press, Cambridge, 1989.

\bibitem{CabCha}
M.~E. Caballero and L.~Chaumont.
\newblock Weak convergence of positive self-similar {M}arkov processes and
  overshoots of {L}\'evy processes.
\newblock {\em Ann. Probab.}, 34(3):1012--1034, 2006.

\bibitem{CabRiv}
Mar{\'\i}a~Emilia {Caballero} and V{\'\i}ctor~Manuel {Rivero}.
\newblock {On the asymptotic behaviour of increasing self-similar Markov
  processes.}
\newblock {\em {Electron. J. Probab.}}, 14:865--894, 2009.

\bibitem{CPY}
Philippe Carmona, Fr{\'e}d{\'e}rique Petit, and Marc Yor.
\newblock On the distribution and asymptotic results for exponential
  functionals of {L}\'evy processes.
\newblock In {\em Exponential functionals and principal values related to
  Brownian motion}, Bibl. Rev. Mat. Iberoamericana, pages 73--130. Rev. Mat.
  Iberoamericana, Madrid, 1997.

\bibitem{CKPR}
Lo{\"{\i}}c Chaumont, Andreas Kyprianou, Juan~Carlos Pardo, and V{\'{\i}}ctor
  Rivero.
\newblock Fluctuation theory and exit systems for positive self-similar
  {M}arkov processes.
\newblock {\em Ann. Probab.}, 40(1):245--279, 2012.

\bibitem{ChYo}
Lo\"{\i}c {Chaumont} and Marc {Yor}.
\newblock {\em {Exercises in probability. A guided tour from measure theory to
  random processes via conditioning. 2nd ed.}}
\newblock Cambridge: Cambridge University Press, 2nd ed. edition, 2012.

\bibitem{Doney}
Ronald~A. Doney.
\newblock {\em Fluctuation theory for {L}\'evy processes}, volume 1897 of {\em
  Lecture Notes in Mathematics}.
\newblock Springer, Berlin, 2007.
\newblock Lectures from the 35th Summer School on Probability Theory held in
  Saint-Flour, July 6--23, 2005, Edited and with a foreword by Jean Picard.

\bibitem{ErMa}
K.~Bruce Erickson and Ross~A. Maller.
\newblock Generalised {O}rnstein-{U}hlenbeck processes and the convergence of
  {L}\'evy integrals.
\newblock In {\em S\'eminaire de {P}robabilit\'es {XXXVIII}}, volume 1857 of
  {\em Lecture Notes in Math.}, pages 70--94. Springer, Berlin, 2005.

\bibitem{HaRi}
B\'en\'edicte Haas and V\'{\i}ctor Rivero.
\newblock Quasi-stationary distributions and {Y}aglom limits of self-similar
  {M}arkov processes.
\newblock {\em Stochastic Process. Appl.}, 122(12):4054--4095, 2012.

\bibitem{Kevei}
P\'eter Kevei.
\newblock Ergodic properties of generalized {O}rnstein-{U}hlenbeck processes.
\newblock {\em Stochastic Processes and their Applications}, pages~--, 2017.

\bibitem{KPS}
A.~Kuznetsov, J.~C. Pardo, and M.~Savov.
\newblock Distributional properties of exponential functionals of {L}\'evy
  processes.
\newblock {\em Electron. J. Probab.}, 17:no. 8, 35, 2012.

\bibitem{La1}
John Lamperti.
\newblock Semi-stable stochastic processes.
\newblock {\em Trans. Amer. Math. Soc.}, 104:62--78, 1962.

\bibitem{La2}
John Lamperti.
\newblock Semi-stable {M}arkov processes. {I}.
\newblock {\em Z. Wahrscheinlichkeitstheorie und Verw. Gebiete}, 22:205--225,
  1972.

\bibitem{Lee}
Oesook Lee.
\newblock Exponential ergodicity and $\beta$-mixing property for generalized
  {O}rnstein-{U}hlenbeck processes.
\newblock {\em Theoretical Economics Letters}, 2(1):21--25, 2012.

\bibitem{LinMa}
Alexander Lindner and Ross Maller.
\newblock L\'evy integrals and the stationarity of generalised
  {O}rnstein-{U}hlenbeck processes.
\newblock {\em Stochastic Process. Appl.}, 115(10):1701--1722, 2005.

\bibitem{MauZw}
Krishanu Maulik and Bert Zwart.
\newblock Tail asymptotics for exponential functionals of {L}\'evy processes.
\newblock {\em Stochastic Process. Appl.}, 116(2):156--177, 2006.

\bibitem{MeynT}
S.~P. Meyn and R.~L. Tweedie.
\newblock {\em Markov chains and stochastic stability}.
\newblock Communications and Control Engineering Series. Springer-Verlag London
  Ltd., London, 1993.
\newblock Also available for download from \texttt{http://probability.ca/MT/}
  free of charge.

\bibitem{ParRiv}
Juan~Carlos Pardo and V\'{\i}ctor Rivero.
\newblock Self-similar {M}arkov processes.
\newblock {\em Bol. Soc. Mat. Mexicana (3)}, 19(2):201--235, 2013.

\bibitem{Patie}
Pierre {Patie}.
\newblock {$q$-invariant functions for some generalizations of the
  Ornstein-Uhlenbeck semigroup.}
\newblock {\em {ALEA, Lat. Am. J. Probab. Math. Stat.}}, 4:31--43, 2008.

\bibitem{Riv}
V\'{\i}ctor Rivero.
\newblock Recurrent extensions of self-similar {M}arkov processes and
  {C}ram\'er's condition.
\newblock {\em Bernoulli}, 11(3):471--509, 2005.

\bibitem{RW}
L.~C.~G. Rogers and David Williams.
\newblock {\em Diffusions, {M}arkov processes, and martingales. {V}ol. 1}.
\newblock Wiley Series in Probability and Mathematical Statistics: Probability
  and Mathematical Statistics. John Wiley \& Sons, Ltd., Chichester, second
  edition, 1994.
\newblock Foundations.

\bibitem{Sato}
Ken-iti Sato.
\newblock {\em L\'evy processes and infinitely divisible distributions},
  volume~68 of {\em Cambridge Studies in Advanced Mathematics}.
\newblock Cambridge University Press, Cambridge, 2013.
\newblock Translated from the 1990 Japanese original, Revised edition of the
  1999 English translation.

\bibitem{Toto}
Haruo Totoki.
\newblock Time changes of flows.
\newblock {\em Mem. Fac. Sci. Kyushu Univ. Ser. A}, 20:27--55, 1966.

\bibitem{Wang}
Jian Wang.
\newblock On the exponential ergodicity of {L}\'evy driven
  {O}rnstein-{U}hlenbeck processes.
\newblock {\em Journal of Applied Probability}, 49(4):990--1004, 2012.

\end{thebibliography}
\end{document}